\documentclass{article}
\usepackage[fontsize=12pt]{scrextend}
\usepackage{euscript, amsmath,amssymb,amsfonts,mathrsfs,amsthm,mathtools,graphicx, tikz, xcolor,verbatim, bm, enumerate, enumitem,babel,multicol}
\usepackage[all]{xy}
\usetikzlibrary{arrows,patterns}
\usepackage{authblk}

\usepackage[latin1]{inputenc}
\usepackage{verbatim}
\usepackage{bm}
\usepackage[justification=centering]{subcaption}

\usepackage{color}
\usepackage[latin1]{inputenc}

\tikzstyle{square} = [shape=regular polygon, regular polygon sides=4, minimum size=1cm, draw, inner sep=0, anchor=south, fill=gray!30]
\tikzstyle{squared} = [shape=regular polygon, regular polygon sides=4, minimum size=1cm, draw, inner sep=0, anchor=south, fill=gray!60]

\newtheorem{theorem}{Theorem}[section]
\newtheorem{definition}[theorem]{Definition}

\newtheorem{lemma}[theorem]{Lemma}
\newtheorem{coro}[theorem]{Corollary}

\newtheorem{prop}[theorem]{Proposition}

\newcommand{\C}{{\mathbb{C}}}
\newcommand{\Z}{{\mathbb{Z}}}

\newcommand{\N}{{\mathbb{N}}}

\newcommand{\Nm}{{\mathrm{Nm}}}

\renewcommand{\Re}{{\text{Re}}}
\renewcommand{\Im}{{\text{Im}}}

\renewcommand{\Re}{{\text{Re}}}
\renewcommand{\Im}{{\text{Im}}}

\begin{document}

\title{A Division Algorithm for the Gaussian Integers' Minimal Euclidean Function}
\author{Hester Graves}
\affil{IDA/Center for Computing Sciences}
\date{\today}

\maketitle

\begin{abstract}The usual division algorithms on $\Z$ and $\Z[i]$ measure the size of remainders using the norm function.
These rings are Euclidean with respect to several functions.
The pointwise minimum of all Euclidean functions $f: R \setminus 0 \rightarrow \N$  on a Euclidean domain $R$ is itself a Euclidean function, 
called the minimal Euclidean function and denoted by $\phi_R$.
The integers, $\Z$, and the Gaussians, $\Z[i]$, are the only rings of integers of number fields for which we have a formula to compute their minimal Euclidean functions, $\phi_{\Z}$ and $\phi_{\Z[i]}$.
This paper presents the first division algorithm for $\Z[i]$ relative to  $\phi_{\Z[i]}$, 
empowering readers to perform the Euclidean algorithm on $\Z[i]$ using its minimal Euclidean function.
\end{abstract}

\section{Introduction}

We call the domain $\Z[i] = \{x +yi: x, y \in \Z, i^2 = -1 \}$ the Gaussian integers because 
Gauss showed they are Euclidean for the (algebraic) norm,
 $\Nm(x +yi) = x^2 +y^2$.
He discovered a division algorithm that, given $a, b \in \Z[i] \setminus \{0\}$, provides $q, r \in \Z[i]$ such that $a =qb +r$ and $\Nm(r) < \Nm(b)$.
Accordingly, we call $q$ the Gauss quotient and $r$ the Gauss remainder.

Inspired by Zariski, Motzkin \cite{Motzkin} broadened the study of Euclidean domains via Euclidean functions.
A domain $R$ is Euclidean if there exists a \textbf{Euclidean function} $f: R \setminus \{ 0 \} \rightarrow W$, 
where $W$ is a well-ordered set with $\N$ as an initial segment, such that for all $a, b \in R \setminus \{0\}$,
there exist some $q, r \in R$ such that $a = qb +r$, where either $f(r) < f(b)$ or $r =0$.
Using this modern terminology, Gauss showed the norm is a Euclidean function for $\Z[i]$.

Motzkin \cite{Motzkin} further showed that if $F$ is the set of all Euclidean functions on $R \setminus \{0\}$, then $\phi_R(x) = \min_{f \in F} f(x)$
is itself a Euclidean function.
For obvious reasons, we call $\phi_R$ the \textbf{minimal Euclidean function} on $R$.  
In particular, he showed $\phi_{\Z}(x) = \lfloor \log_2 |x| \rfloor $.
Until 2023, this was the only formula to compute the minimal Euclidean function for any number field.

Just as every integer has a binary expansion, every Gaussian integer has $(1+i)$-ary expansions.
This is fitting, as $2 =-i(1+i)^2$ and the quotient $\Z[i]/ \langle 1+i \rangle$ has size $2$.
We use $(1+i)$-expansions of the form $x +yi = \sum_{j=0}^n u_j (1+i)^j$ for some $n$, where $u_j \in \{\pm 1, \pm i, 0\}$ and $u_n \neq 0$.
These expressions are not unique, as 
\begin{equation}\label{2+i}
 2 +i = -i(1+i)^2 +i = (1+i) +1.
\end{equation}
Lenstra (\cite{Lenstra}, section 11) showed that $\phi_{\Z[i]}(x+yi)$ is the minimal degree of all $(1+i)$-ary expansions of $x+yi$, e.g., 

The author's recent research gives an explicit formula for $\phi_{\Z[i]}(x+yi)$, using valuations and the sequence $\{w_m\}$ \cite{Graves}.
If $a$ divides $b$, we write $a|b$.
Otherwise, $a \nmid b$.
When $a^c |b$ but $a^{c+1} \nmid b$, we say $c$ is the $a$-valuation of $b$, or $v_a(b) =c$.
We define the sequence 
\[ w_m = \begin{cases} 3 \cdot 2^k & \text{ if } m = 2k \\ 4 \cdot 2^k & \text{ if } m = 2k+1 \end{cases}, \]
so that $w_0 = 3$, $w_1 =4$, $w_2 = 6$, $w_3 = 8$, etc.  
We rely on the facts that $w_{m+2} = 2 w_m$ and that if $2^l < w_n$, then $2^l |w_m$ for all $m >n$.
We use $\gcd(x,y)$ to denote the greatest common divisor of $x$ and $y$.

\begin{theorem}(Graves, \cite{Graves})\label{formula}
Given $x +yi \in \Z[i] \setminus \{0\}$, let $j = v_2(\gcd(x,y))$ and $n$ be the smallest integer such that 
$\frac{|x|}{2^j}, \frac{|y|}{2^j} \leq w_n -2$.
If $\frac{|x| + |y|}{2^j} \leq w_{n+1} -3$, then $\phi_{Z[i]}(x+yi) = n+2j$.
Otherwise, $\phi_{\Z[i]}(x+yi) = n+2j+1$.
\end{theorem}

As a consequence, if $x$ and $y$ are $N$-bit integers, we can compute $\phi_{\Z[i]}(x+yi)$ in time $O(\log N)$. 
In fact, section 4 of \cite{Graves} shows how to compute minimal $(1+i)$-ary expansions in time $O(\log N)$.

Stating the theorem's corollary below requires some more terminology.
Every complex number has a real and imaginary part, which we denote by $\Re(x+yi) = x$ and $\Im(x +yi) = y$.
The $\ell_1$-norm is $\ell_1(x +yi) = |x| + |y|$ and the $\ell_{\infty}$-norm is $\ell_{\infty}(x +yi) = \max (|x|, |y|)$.
For ease of notation, we use the following definitions, including the function 
\[ m(x+yi) = \ell_1(x +yi) - \ell_{\infty}(x +yi)  = \min (|x|, |y|).\]
\begin{definition}
If $z \in \C \setminus \{0\}$ and $\ell_{\infty}(z) \neq m(z)$, there exists a unique unit $u_{z}$ 
such that $\Re(u_{z} z) = \ell_{\infty}(z)$.
If $z \in  \C \setminus \{0 \}$ and $\ell_{\infty} (z) = m(z)$, there exists a unique unit $u_{z}$ 
such that $u_{z} z = \ell_{\infty} (z) (1 +i)$.
\end{definition}
For example, $u_{-1 +2i} = -i$ and $u_{1-i} = i$. 

\begin{coro}\label{coro_fcn_bounds}(of Theorem \ref{formula}) If $z \in \Z[i] \setminus \{0\}$, 
if $\ell_{\infty}(z) \leq w_n - 2^{v_2(z) +1}$,
and if $\ell_1(z) \leq w_{n+1} - 3 \cdot 2^{v_2(x)}$, then $\phi_{\Z[i]}(z) \leq n$.
Note $\phi_{\Z[i]}(z) > n-1$ if and only if either $\ell_{\infty}(z) > w_{n-1} - 2^{v_2(z) +1}$ or 
$\ell_1(z) > w_n - 3 \cdot 2^{v_2(z)}.$
We also note that if $\phi_{\Z[i]}(z) \leq n$, then $m(z) \leq w_{n-1} - 2^{v_2(z) +1}$.
\end{coro}

Given $a, b \in \Z[i] \setminus \{0\}$, Theorem \ref{main_result} shows how to adjust $a$ and $b$'s Gauss remainder to find new quotients and remainders
$q, r \in \Z[i]$ such that $a = qb +r$, where either $r =0$ or $\phi_{\Z[i]}(r) < \phi_{\Z[i]}(b)$.
Interestingly, there are times we do not have to compute the Gauss remainder to know whether we need to adapt it for $\phi_{\Z[i]}$. 
We need the following definition to state Theorem \ref{main_result}.

\begin{definition} For $x \in \mathbb{R}$ and $r \in \mathbb{C}$, we denote 
\[
\text{sgn}(x) = \begin{cases} 1 & \text{ if } x >0 \\ 0 & \text{ if } x =0 \\ -1 & \text{ if } x <0 \end{cases}
\hspace{1.5 cm} 
\text{and}
\hspace{ 1.5 cm} 
s(r) = \text{sgn} (\Im(u_r r)).
\]
\end{definition}

\begin{theorem}\label{main_result}  Suppose $a,b \in \Z[i] \setminus \{0\}$ have Gauss quotient $q$ and non-zero Gauss remainder $r$ with 
$\phi_{\Z[i]}(r) \geq \phi_{\Z[i]}(b ) =n$.
If any of the following conditions hold
\begin{enumerate}
\item $\Im(u_b b) \Im (u_r r) \geq 0$,\\
\item 
$m(r) + m(b) \leq \ell_{\infty}(r)$,\\
\item 
$m(b) < \ell_{\infty}(r) < m(b) + m(r)$ and $\ell_{\infty}(b) - m(r) > w_{n-1} - 2^{v_2(b) +1}$,
\end{enumerate} then 
$a = \left (q + \frac{u_b}{u_r} \right ) b + \left ( r - \frac{u_b}{u_r} b \right )$, with $\phi_{\Z[i]} \left ( r - \frac{u_b}{u_r} b \right ) < \phi_{\Z[i]} (b)$.
Otherwise, 
$a = \left ( q + \frac{iu_b}{s(r) u_r} \right ) + \left ( r - \frac{i u_b}{s(r) u_r} b \right )$, with $\phi_{\Z[i]} \left ( r - \frac{i u_b}{s(r) u_r} b \right ) < \phi_{\Z[i]} (b)$.
\end{theorem}

\section{Background}
Given $z \in \C$, its complex conjugate is $\overline{z} = \Re(z) - \Im(z)i$.
The algebraic norm of a complex number is the product of itself and its conjugate, i.e.,
\[\Nm(x +yi) = ( x +yi)(x -yi) = x^2 + y^2.\]
Complex conjungation is multiplicative, so that $\overline{xy} = \overline{x} \cdot \overline{y}$, 
and hence the norm is also multiplicative.  
In other words, if $a$ and $b$ are complex numbers, then $\Nm(ab) = \Nm(a) \Nm(b)$.
The Gaussians' multiplicative units, $\Z[i]^{\star} = \pm \{1, i\}$, have norm one, so 
both the norm and $\phi_{\Z[i]}$ are invariant under multiplication by units.  
If $u \in \Z[i]^{\times}$ and $z \in \Z[i] \setminus \{0\}$, then $\Nm(z) = \Nm(uz)$ and 
$\phi_{\Z[i]}(z) = \phi_{\Z[i]}(uz)$.

Suppose $z \in \Z[i] \setminus \{0\}$.  If $v_{1+i}(z) \geq 2$, then $2 | z$ and thus $\Re(z)$ and $\Im(z)$ are both even.
If $v_{1+i}(z) =1$, then $v_{1+i}(z - (1+i)) \geq 2$ and $\Re(z)$ and $\Im(z)$ are both odd.
We observe that when $v_{1+i}(z) =0$, exactly one of $\Re(z)$ and $\Im(z)$ is even, forcing oddity upon $\ell_1(z)$ and $\ell_{\infty}(z)-m(z)$. 

If $x$ is a real number and $n$ is the integer such that $n \leq x < n+1$, we write $\lfloor x \rfloor =n$ and $\lceil x \rceil = n+1$.
We also use the nearest integer function, given by nint$(x)$ or 
\[ \lfloor x \rceil = 
\begin{cases} \lfloor x \rfloor & \text{ if } 0 \leq x - \lfloor x \rfloor \leq 1/2 \\ \lceil x \rceil & \text{ otherwise} \end{cases}
.\]


\section{Gauss's division algorithm}

Suppose $a,b \in \Z[i] \setminus \{0\}$, with $a = x +yi$ and $b = c +di$.
Then $\frac{a}{b}$ equals 
\begin{align*}
\frac{a \overline{b}}{b \overline{b}} & = \frac{a \overline{b}}{\Nm(b)} = \frac{(xc -yd) + (xd+yc)i}{\Nm(b)}.\\
\intertext{If $q_0 = \left \lfloor \frac{xc -yd}{\Nm(b)} \right \rceil $, $q_1 = \left \lfloor \frac{xd +yc}{\Nm(b)} \right \rceil $,
$f_0 = \frac{xc-yd}{\Nm(b)} -q_0$, and $f_1 = \frac{xd +yc}{\Nm(b)} -q_1$, then}
\frac{a\overline{b}}{\Nm(b)} &= (q_0 + q_1 i) + (f_0 + f_1 i), \text{ where } |f_0|, |f_1| \leq \frac{1}{2}.\\
\intertext{Multiplying through by $b$ shows}
a & = (q_0 + q_1 i)b + (f_0 + f_1 i) b,\\
\intertext{where $q_0 + q_1 i, (f_0 + f_1 i)b \in \Z[i]$.  The norm is multiplicative, so}
\Nm((f_0 + f_1i)b) &= \Nm(f_0 + f_1 i) \Nm(b) \leq \left ( \left (\frac{1}{2} \right )^2 + \left (\frac{1}{2} \right )^2 \right ) \Nm(b) \leq \frac{\Nm(b)}{2}.\\
\end{align*}
We call $q_0 + q_1 i$ the \textbf{Gauss quotient}, $f_0 + f_1 i$ the \textbf{Gauss fractional remainder}, and $(f_0 + f_1i)b$ the \textbf{Gauss remainder}.
This algorithm has straightforward consequences.

\begin{lemma}\label{remainder_zr}  Suppose $a,b, z \in \Z[i] \setminus \{0\}$.  The pair $a$ and $b$ have Gauss quotient $q$ and Gauss remainder $r$
 if and only if $za$ and $zb$ have Gauss quotient $q$ and Gauss remainder $zr$.
\end{lemma}
\begin{proof}
The fraction $\frac{za}{zb} \cdot \frac{\overline{zb}}{\overline{zb}}$ simplifies to $\frac{a}{b}$, so $za$ and $zb$ have the same Gauss quotient and fractional
remainder as $a$ and $b$.  This means that the Gauss remainder of $za$ and $zb$ is $za-qzb =zr$, or $z$ times the Gauss remainder of $a$ and $b$.
\end{proof}

\begin{lemma}\label{bounds_lemma}  If $a, b \in \Z[i] \setminus \{0\}$ have Gauss remainder $r$ and if $\phi_{\Z[i]}(b) =n$,
 then $\ell_1(r) \leq \ell_{\infty}(b) < w_{n}$ and $\ell_{\infty}(r) \leq \frac{\ell_1(b)}{2}< w_{n-1}$.
Equality occurs only if $\frac{r}{b} \in \pm \frac{1}{2} \{ 1, i, 1 \pm i\}.$
\end{lemma}
\begin{proof}
There exist units $u, w \in \Z[i]^{\times}$ such that $\Re(ub)$, $\Im(ub)$, $\Re\left ( w \frac{r}{b} \right )$, 
$\Im \left ( w \frac{r}{b} \right ) \geq 0$.
For ease of notation, let $ub = x +yi$ and $w \frac{r}{b} = f_0 + f_1 i$.
Then $r = (xf_0 -yf_1) + (xf_1 + yf_0)i$, so 
\begin{align*}
\ell_1(r) &= \max(|xf_0 - yf_1 + xf_1 + y f_0|,| yf_1 - xf_0 + x f_1 +yf_0|)\\
 &= \max (|f_0(x+y) + f_1(x-y)|,| f_0(y-x) +f_1(x+y)|)\\
 &\leq \frac{\ell_1(b)}{2} + \frac{\ell_{\infty}(b) - m(b)}{2} = \ell_{\infty}(b) < w_n\\
 \intertext{and}
 \ell_{\infty}(r) &= \max(|xf_0 - yf_1|, |xf_1 + y f_0|) \leq \frac{\ell_1(b)}{2} < \frac{w_{n+1}}{2} = w_{n-1}.
 \end{align*}
 These inequalities are strict, unless $f_0, f_1 \in \left \{0, \frac{1}{2} \right \}$.
 \end{proof}
 
\begin{coro} \label{coro_strictly_less} If $a,b \in \Z[i] \setminus \{0\}$ have Gauss remainder $r$ and $\phi_{\Z[i]}(r) \geq \phi_{\Z[i]}(b)$,
then $\ell_{\infty}(r) < \frac{\ell_1(b)}{2}$ and $\ell_1(r) < \ell_{\infty}(b)$.
\end{coro}
\begin{proof} 
If $f_0, f_1 \in \left \{0, \frac{1}{2} \right \}$, then $\frac{r}{b} \in \Z[i]^{\times} \{(1+i)^{-1}, (1+i)^{-2} \}$
and $\phi_{\Z[i]}(r) \in \{ \phi_{\Z[i]}(b) -1, \phi_{\Z[i]}(b) -2 \}$.
This means that if $\phi_{\Z[i]}(r) \geq \phi_{\Z[i]}(b)$, our inequalities are strict.
\end{proof}

\section{Gauss remainders and the minimal Euclidean function}

Applying Lemma \ref{bounds_lemma} shows the Gauss remainder and quotient often work for $\phi_{\Z[i]}$.

\begin{lemma}\label{2_val_lemma} If $a,b \in \Z[i] \setminus \{0\}$ have Gauss remainder $r$ and $v_2(r) \leq v_2(b)$,
then $\phi_{\Z[i]}(r) < \phi_{\Z[i]}(b)=n$.
\end{lemma}
\begin{proof} 
Lemma \ref{bounds_lemma} shows that if $\ell_1(r) = \ell_{\infty}(b)$, then 
$\frac{r}{b} \in \pm \frac{1}{2} \{1, i, 1 \pm i \}$ and thus $\phi_{\Z[i]}(r) < \phi_{\Z[i]}(b)$.
We therefore assume 
\begin{align*}
\ell_1(r) & < \ell_{\infty}(b) \leq w_n - 2^{v_2(b) +1} \leq w_n - 2^{v_2(r) +1}\\
\intertext{and}
2^{v_2(r)} &\leq \ell_{\infty}(r) < \frac{\ell_1(b)}{2} \leq \frac{ w_{n+1} - 3 \cdot 2^{v_2(r)}} {2}.
\end{align*}
The second inequality shows $2^{v_2(r)} | w_{n-1}$, $3 \cdot 2^{v_2(r)} \leq w_{n-1}$, 
 $\ell_{\infty}(r) \leq w_{n-1} - 2^{v_2(r) +1}$, and 
$\ell_1(r) \leq w_n - 3 \cdot 2^{v_2(r)}$.
We conclude $\phi_{\Z[i]}(r) \leq n-1 < \phi_{\Z[i]}(b)$.
\end{proof}
We do not even always need to compute the Gauss remainder to determine whether it is a remainder for $\phi_{\Z[i]}$.
\begin{coro}\label{not_even} If $a,b \in \Z[i] \setminus \{0\}$ have Gauss remainder $r$ and $v_2(a) \leq v_2(b)$,
then $\phi_{\Z[i]}(r) < \phi_{\Z[i]}(b)$.
\end{coro}



\section{Constructing Alternate Remainders}

Unlike in Corollary \ref{not_even}, computing a remainder for $a,b \in \Z[i] \setminus \{0\}$ and $\phi_{\Z[i]}$ is not so cut and dry when
the two-valuation of the Gauss remainder $r$ is greater than $v_2(b)$.
Observe that $v_{1+i}(4+i) =0$, that $\phi_{\Z[i]}(2i) = 2 = \phi_{\Z[i]}(4+i)$,
and that the Gauss remainder of $9$ and $4+i$ is $2i$, so $2i$ is not the pair's remainder for $\phi_{\Z[i]}$.
In this section, we construct alternate remainders, for when $\phi_{\Z[i]}(r) \geq \phi_{\Z[i]}(b)$.
The new remainder determines the new quotient.
If $R$ is the new remainder, then the associated quotient is the integer $\frac{a-R}{b}$.  
Any potential new remainder $R$ is an element of the set $r + b \Z[i]$.

\subsection{Valuations and Preliminaries}

It is worth noting that if $\phi_{\Z[i]}(b) =n$, then $2^{v_2(b)} \leq \ell_{\infty}(b) \leq w_{n} - 2^{v_2(b) +1}$.  
Since $2^{v_2(b) + 1 } < 
w_n$, we see that $2^{v_2(b) +1} |w_m$ for all $m \geq n+1$.
The bound $2m(b) \leq \ell_1(b) \leq w_{n+1} - 3 \cdot 2^{v_2(b)}$ therefore implies 
$m(b) \leq w_{n-1} - 2^{v_2(b) +1}$. 

\begin{lemma}\label{diff_is_less} If $a,b \in \Z[i] \setminus \{0\}$ have Gauss remainder $r$ and $\phi_{\Z[i]}(r) \geq \phi_{\Z[i]}(b)$,
then $\ell_{\infty}(b) - m(b) \leq w_n - 3 \cdot 2^{v_2(b)}$.
\end{lemma} 
\begin{proof} Suppose, leading to contradiction, that $\ell_{\infty}(b) - m(b) > w_n - 3 \cdot 2^{v_2(b)}$,
forcing $\ell_{\infty}(b) = w_n - 2^{v_2(b) +1}$ and $m(b) =0$.
Since $2^{v_2(b) +1} \nmid \ell_{\infty}(b)$, we see that $w_n = 3 \cdot 2^{v_2(b)}$ and $\ell_{\infty}(b) = 2^{v_2(b)}$.
Lemma \ref{2_val_lemma}'s  reveals a contradiction, as 
\[2^{v_2(b)} < 2^{v_2(r)} \leq \ell_1(r) < \ell_{\infty}(b) = 2^{v_2(b)}.\]
\end{proof}

\begin{lemma}\label{lemma_2_val_nmid}  If $a, b \in \Z[i] \setminus \{0\}$ have Gauss remainder $r \neq 0$, if 
$2^{v_2(r)} \nmid w_{n-1}$, and if $\phi_{\Z[i]}(r) \geq \phi_{\Z[i]}(b) =n$, and , then $m(r) =0$ and $\phi_{\Z[i]} \left  ( r - \frac{u_b}{u_r} b \right ) < \phi_{\Z[i]}(b)$.
\end{lemma}
\begin{proof}
Since $2^{v_2(r)} \nmid w_{n-1}$,  and lemma \ref{bounds_lemma} implies $2^{v_2(r)} \leq \ell_{\infty}(r) < w_{n-1}$, we infer $w_n = 2^{v_2(r) +1}$. 
It also shows $2^{v_2(r)} \leq \ell_{\infty}(r) \leq  \ell_1(r) < w_n = 2^{v_2(r) +1}$, 
forcing $ \ell_{\infty}(r) = \ell_1(r)$ and $m(r) =0$.
This means 
\begin{align*}
u_b b - u_r r &\in \{ (\ell_{\infty}(b) - \ell_{\infty}(r)) \pm m(b) i\},\\
\intertext{so}
\ell_{\infty}(u_b b - u_r r) &\leq \max (\ell_{\infty}(b) - \ell_{\infty}(r), m(b))\\
&\leq ( w_n - 2^{v_2(b) + 1} - 2^{v_2(r)}, w_{n-1} - 2^{v_2(b) +1})\\
&\leq w_{n-1} - 2^{v_2(b) +1}\\
\intertext{by Lemma \ref{2_val_lemma} and}
\ell_1(u_b b - u_r r) &= \ell_1(b) - \ell_{\infty}(r)\\
&\leq w_{n+1} - 3 \cdot 2^{v_2(b)} - 2^{v_2(r)}\\
&= w_n - 3 \cdot 2^{v_2(b)}.
\end{align*}
Since 
$v_2(u_b b - u_r r) = v_2(b)$, we deduce
 $\phi_{\Z[i]} \left ( r - \frac{u_b}{u_r} b \right ) = \phi_{\Z[i]}( u_b b - u_r r) < \phi_{\Z[i]} (b)$.
\end{proof}

Since we now understand what happens when $\phi_{\Z[i]}(r) \geq \phi_{\Z[i]}(b)$ and $2^{v_2(r)} \nmid w_{n-1}$, 
we want to study the possibilities when $2^{v_2(r)} | w_{n-1}$.  

\begin{lemma}\label{lemma_remainder_properties} Suppose  $a, b \in \Z[i] \setminus \{0 \}$ have Gauss remainder $r \neq 0$,
that $\phi_{\Z[i]}(r) \geq \phi_{\Z[i]}(b) =n$, and that $2^{v_2(r)} | w_{n-1}$.
Then $2^{v_2(r)}| w_m$ for all $m \geq n-1$ and either 
\begin{align*}
\ell_{\infty}(r) = w_{n-1} - 2^{v_2(r)} &\text{ or } \ell_1(r) \geq w_n - 2^{v_2(r) +1}.\\
\intertext{In both cases,}
\ell_{\infty}(r) \geq w_n - w_{n-1} &\text{ and } \ell_1(r) \geq w_{n-1} - 2^{v_2(r)}.
\end{align*}
If $\ell_{\infty}(r) \neq w_{n-1} - 2^{v_2(r)}$, 
then $m(r) \geq w_n - w_{n-1}$.
\end{lemma}

\begin{proof} Lemma \ref{bounds_lemma} shows $2^{v_2(r)} \leq \ell_{\infty}(r) < w_{n-1}$, so 
$\ell_{\infty}(r) \leq w_{n-1} - 2^{v_2(r)}$ and $2^{v_2(r) + 1 }\leq w_{n-1}$.  
It follows that $2^{v_2(r)} | w_m$ for all $m \geq n-1$, and thus Corollary \ref{coro_fcn_bounds} 
implies either $\ell_{\infty}(r) =  w_{n-1} - 2^{v_2(r) }$ or $\ell_1(r) \geq w_n -2^{v_2(r) +1}$.

If $\ell_{\infty(r)} \leq w_{n-1} - 2^{v_2(r) +1}$, then 
$\ell_1(r) =\ell_{\infty}(r) + m(r) \geq w_n - 2^{v_2(r) +1}$
implies
\[m(r) \geq w_n - 2^{v_2(r) + 1} - (w_{n-1} -2^{v_2(r)+1}) = w_n - w_{n-1}.\]
Note that since $2^{v_2(r)}$ divides both $w_{n-1}$ and $w_n$,
$w_n - 2^{v_2(r) +1} \geq w_{n-1} - 2^{v_2(r)}$,
and thus $\ell_1(r) \geq w_n - 2^{v_2(r) +1}$, no matter which of our two conditions hold.
Similarly, 
$w_{n-1} - 2^{v_2(r)} \geq w_n - w_{n-1}$ and thus 
$\ell_{\infty}(r) \geq w_n - w_{n-1}$ in both situations.
\end{proof}

\begin{coro}\label{coro_big_diff}
If $a, b \in \Z[i] \setminus \{0 \}$ have non-zero Gauss remainder $r$,
if $\phi_{\Z[i]}(r) \geq \phi_{\Z[i]}(b) =n$, and if $2^{v_2(r)} | w_{n-1}$,
then $m(r), \ell_{\infty}(r), \ell_{\infty}(b) - \ell_{\infty}(r) \leq w_{n-1} - 2^{v_2(b) +1}$.
\end{coro}
\begin{proof} 
Lemma \ref{2_val_lemma} implies $v_2(r) > v_2(b)$, so Lemma \ref{lemma_remainder_properties} shows 
$m(r), \ell_{\infty}(r) \leq w_{n-1} - 2^{v_2(r)} \leq w_{n-1} - 2^{v_2(b)+1}$.
Since either 
$\ell_{\infty}(r) = w_{n-1} - 2^{v_2(r)}$ or $\ell_{\infty}(r)  \geq w_n - w_{n-1}$, 
\begin{align*}
\ell_{\infty}(b) - \ell_{\infty}(r) &\leq \max(\ell_{\infty}(b) - w_{n-1} + 2^{v_2(r)}, \ell_{\infty}(b) -w_n  + w_{n-1})\\
&\leq \max(w_n - 2^{v_2(b) +1} - w_{n-1} + (w_{n+1} - w_n), w_n - 2^{v_2(b) +1} - w_n + w_{n-1})\\
&=w_{n-1} - 2^{v_2(b) + 1}.
\end{align*}

\end{proof}

\subsection{When imaginary parts align}

Determining an alternate remainder is fairly straightforward when $\Im(u_b b) \Im(u_r r) \geq 0$, i.e., when $u_b b$ and $u_r r$ lie in the same quadrant.  
The next section shows that things become much more complicated when they do not.

\begin{prop}\label{im_align}  Suppose $a, b \in \Z[i] \setminus \{0\}$ have non-zero Gauss remainder $r$.
If $\phi_{\Z[i]}(r) \geq \phi_{\Z[i]}(b)=n$ and $\Im(u_b b) \Im (u_r r) \geq0$ then 
$\phi_{\Z[i]} \left ( r - \frac{u_b}{u_r} b \right ) < \phi_{\Z[i]}(b)$.
\end{prop}
\begin{proof}
Lemma \ref{lemma_2_val_nmid} 
lets us assume 
$2^{v_2(r)} |w_m$ for all $m \geq n-1$, so
\begin{align*}
u_b b - u_r r &\in \{(\ell_{\infty}(b) - \ell_{\infty}(r)) \pm (m(b) - m(r))i \},\\
\intertext{and thus}
\ell_{\infty}(u_b b - u_r r) &\leq \max (\ell_{\infty}(b) - \ell_{\infty}(r), m(b), m(r)) \leq w_{n-1} -2^{v_2(b) +1}
\intertext{by Corollary \ref{coro_big_diff}.  Observe that}
\ell_1(u_b b - u_r r) &= \max(\ell_1(b) - \ell_1(r), \ell_{\infty}(b) - m(b) + m(r) - \ell_{\infty}(r))\\
&\leq \max(w_{n+1} - 3 \cdot 2^{v_2(b)} - (w_{n-1} - 2^{v_2(r)}), \ell_{\infty}(b) - m(b))\\
&\leq \max(w_{n-1} -3\cdot 2^{v_2(b)} + (w_n - w_{n-1}) , w_n -3\cdot 2^{v_2(b)}) \\
&= w_n -3\cdot 2^{v_2(b)}.
\end{align*}
Because $v_2(u_b b - u_r r) = v_2(b)$, we conclude
$\phi_{\Z[i]} \left ( r - \frac{u_b}{u_r} b \right ) 
< \phi_{\Z[i]}(b)$.
\end{proof}

\section{Proving our Main Result}

When $\Im(u_b b)$ and $\Im(u_r r)$ have opposite signs, finding alternate remainders becomes more complicated.
It only partially depends on how $r$ relates to $m(b)$.
In this section, we construct alternate remainders for when $\phi_{\Z[i]}(r) \geq \phi_{\Z[i]}(b)$ and\\
 $\Im(u_b b) \Im(u_r r) <0$, allowing us to prove Theorem \ref{main_result}.

\begin{lemma}\label{twisted}
Suppose $a,b \in \Z[i] \setminus \{0\}$ have non-zero Gauss remainder $r$.
If $\phi_{\Z[i]}(r) \geq \phi_{\Z[i]}(b) =n$ and $\Im(u_b b) \Im(u_r r) <0$,
then 
\begin{align*}
u_b b -u_r r &\in \{ (\ell_{\infty}(b) - \ell_{\infty}(r)) \pm (m(b) + m(r))i \} \\
\intertext{and}
u_b b + s(r) iu_r r &\in \{ (\ell_{\infty}(b) - m(r) ) \pm (m(b) - \ell_{\infty}(r)) i \}.
\end{align*}
\end{lemma}
\begin{proof} 
The first part follows from the definitions. 
If $u_b b = \ell_{\infty}(b) \pm m(b) i$ and $u_r r = \ell_{\infty}(r) \mp m(r) i$, then 
$iu_r r = \pm m(r) + \ell_{\infty}(r) i$ and thus $s(r)iu_r r = - m(r) \mp \ell_{\infty}(r) i$.
We conclude that $u_b b + s(r) i u_r r = (\ell_{\infty}(b) -m(r)) \pm (m(b) -\ell_{\infty}(r))i$.
\end{proof}

\begin{lemma}\label{m(b)_leq} Suppose $a,b \in \Z[i]\setminus \{0\}$ have  Gauss remainder $r \neq 0 $.
If $\Im(u_b b) \Im(u_r r) <0$, if $\phi_{\Z[i]}(r) \geq \phi_{\Z[i]}(b) =n$,  and if $m(b) \geq \ell_{\infty}(r)$, then 
$\phi_{\Z[i]}\left ( r - \frac{u_b i}{s(r) u_r} b \right ) < \phi_{\Z[i]}(b)$.  
\end{lemma}
\begin{proof} Since $m(r) \neq 0$,  Lemma \ref{lemma_2_val_nmid} shows that $2^{v_2(r)}$ divides $w_m$ for all $m \geq n-1$. 
Since $m(b) \geq \ell_{\infty}(b)$, Lemma \ref{twisted} shows 
\begin{align*}
\ell_1(u_b b + s(r) iu_r r) &= \ell_1(b) - \ell_1(r).\\
\intertext{Lemma \ref{2_val_lemma} shows $v_2(r) >v_2(b)$, so we infer $v_2(u_b b + s(r) i u_r r) = v_2(b)$.  If $\ell_{\infty}(r) = w_{n-1} - 2^{v_2(r)}$, then $\ell_1(r) \geq w_{n-1}$ and}
\ell_1(u_b b + s(r) i u_r r) &\leq \ell_1(b) - w_{n-1} \leq w_{n-1} -3 \cdot 2^{v_2(b)}.
\end{align*}
It follows that $\phi_{\Z[i]}(u_b b + s(r) iu_r r) < n$.

If $\ell_{\infty}(r) \leq w_{n-1} - 2^{v_2(r) +1}$, then 
$\ell_1(r) \geq w_n - 2^{v_2(r) +1},$
\begin{align*}
\ell_1(u_b b + s(r)iu_r r) 
&\leq w_{n+1} - 3\cdot 2^{v_2(b)} - w_n + 2^{v_2(r) +1}\\
&\leq (w_{n+1} - w_n) + (w_{n+2} - w_{n+1}) -3\cdot 2^{v_2(b)}\\
&= w_n -3\cdot 2^{v_2(b)},\\
\intertext{and}
\ell_{\infty}(u_b b + s(r)i u_r r) &\leq \max (\ell_{\infty}(b) -  m(r),\ell_{\infty}(r), m(b))\\
&\leq \max(w_n - 2^{v_2(b)+1} - (w_n - w_{n-1}), w_{n-1} -2^{v_2(b)+1} )\\
&= w_{n-1} -2^{v_2(b)+1} ,
\end{align*}
proving $\phi_{\Z[i]} \left ( r - \frac{i u_b}{s(r) u_r} b \right ) = 
\phi_{\Z[i]} ( u_b b + s(r) i u_r r) < \phi_{\Z[i]}(b)$.
\end{proof}

\begin{lemma}\label{m(r)+m(b)_leq}
Suppose $a,b \in \Z[i] \setminus \{0\}$ have Gauss remainder $r \neq 0$.
If $\Im (u_b b)\Im(u_r r) <0$, if $\phi_{\Z[i]}(r) \geq \phi_{\Z[i]}(b) =n $, and if $m(r) + m(b) \leq \ell_{\infty}(r)$,
then $\phi_{\Z[i]}\left ( r - \frac{u_b}{u_r} b\right ) < \phi_{\Z[i]}(b)$.
\end{lemma}
\begin{proof} 
Lemmas \ref{2_val_lemma} and \ref{lemma_2_val_nmid} show $v_2(r) > v_2(b)$
and $2^{v_2(b) +1}$ divides $w_m$ for all $m \geq n -1$. 
Lemma \ref{twisted} and Corollary \ref{coro_big_diff} then tell us that  
\begin{align*}
\ell_{\infty}(u_b b - u_r r)&\leq \max (\ell_{\infty}(b) - \ell_{\infty}(r), m(b) + m(r))\\
&\leq \max(w_n - 2^{v_2(b) +1} - (w_n - w_{n-1}), \ell_{\infty}(r))\\
&= w_{n-1} -2^{v_2(b) +1}\\
\intertext{and} 
\ell_1(u_b b - u_r r)&=\ell_{\infty}(b) - \ell_{\infty}(r) + m(b) + m(r)\\
&\leq \ell_{\infty}(b) \leq w_n -2^{v_2(b) +1}.\\
\intertext{If $\ell_{\infty}(r) = m(r) + m(b)$, then $2^{v_2(b)+1} | m(b)$ and $2^{v_2(b)+1} \nmid \ell_{\infty}(b)$.
We deduce that either $\ell_{\infty}(b) < w_n - 2^{v_2(b)+1}$ or $m(b) + m(r) < \ell_{\infty}(r)$, forcing }
\ell_1(u_b b - u_r r)&\leq w_n -3\cdot2^{v_2(b) } .
\end{align*}
We conclude
$\phi_{\Z[i]} \left ( r - \frac{u_b}{u_r} b \right ) = \phi_{\Z[i]}(u_b b - u_r r) < \phi_{\Z[i]}(b)$. 
\end{proof}

We have found alternate remainders when $\ell_{\infty}(r) \leq m(b)$ and when $\ell_{\infty}(r) \geq m(b) +m(r)$,
so we now examine when $m(b) + m(r) > \ell_{\infty}(r) > m(b)$. 

\begin{lemma}\label{ell+m_leq} Suppose $a,b \in \Z[i] \setminus \{0 \}$ have Gauss remainder $r \neq 0$.  
If $\Im(u_b b) \Im (u_r r) <0$, if$\phi_{\Z[i]}(r) \geq \phi_{\Z[i]}(b) =n$, if $m(b) < \ell_{\infty}(r) < m(b) + m(r)$,
and if $\ell_{\infty}(b) - m(r) \leq w_{n-1} - 2^{v_2(b) +1}$, then 
$\phi_{\Z[i]} \left ( r - \frac{i u_b}{s(r) u_r} b \right ) < \phi_{\Z[i]} (b)$. 
\end{lemma}
\begin{proof} 
Lemma \ref{2_val_lemma} shows $v_2(r) > v_2(b)=v_2(u_b b + s(r)i u_r r)$ and,
by Lemma \ref{twisted} and Corollary \ref{coro_big_diff}, we see 
\begin{align*}
\ell_{\infty}(u_b b + s(r) i u_r r) &\leq \max (\ell_{\infty}(b) - m(r), \ell_{\infty}(r))
\leq  w_{n-1} - 2^{v_2(b) +1}\\
\intertext{and}
\ell_1(u_b b + s(r) i u_r r) &= \ell_{\infty}(b) + \left ( \ell_{\infty}(r) - (m(b) + m(r)) \right )
< w_n - 2^{v_2(b) +1}.
\intertext{Lemmas \ref{2_val_lemma} and \ref{lemma_2_val_nmid} show $2^{v_2(b)} | w_n$, so since $2^{v_2(b)} |\ell_1(u_b b + s(r) i u_r r)$,}
\ell_1(u_b b + s(r) i u_r r) &\leq w_n - 3\cdot 2^{v_2(b)}.
\end{align*}
We therefore see that 
$\phi_{\Z[i]}\left ( r - \frac{i u_b}{s(r) u_r} b \right ) = \phi_{\Z[i]} (u_b b +s(r) i u_r r) < \phi_{\Z[i]}(b)$.
\end{proof}

\begin{lemma}\label{ell+m>} Suppose $a,b \in \Z[i] \setminus \{0\}$ have non-zero Gauss remainder $r$.
If $\phi_{\Z[i]}(r) \geq \phi_{\Z[i]}(b) =n$, if $\Im(u_b b) \Im (u_r r) <0$, if $m(b) < \ell_{\infty}(r) < m(r) + m(b)$, 
and if $\ell_{\infty}(b) - m(r) > w_{n-1} - 2^{v_2(b) +1}$, then
$\phi_{\Z[i]} \left (r - \frac{u_b}{u_r} b \right ) < \phi_{\Z[i]}(b)$.
\end{lemma}
\begin{proof}
For the last time, Lemmas \ref{2_val_lemma} and \ref{lemma_2_val_nmid} show that
$2^{v_2(r)} | w_m$ for all $m \geq n-1$ and
 $v_2(r) > v_{2}(b) =v_2(u_b b - u_r r)$.
Note that 
\begin{equation}\label{long_equation}
\ell_{\infty}(b) - (w_n - w_{n-1}) \leq (w_n -2^{v_2(b) +1}) -(w_n - w_{n-1}) = w_{n-1} -2^{v_2(b) +1},
\end{equation}
so our assumptions imply 
$m(r) \leq w_n - w_{n-1} - 2^{v_2(r)}.$
Lemma \ref{lemma_remainder_properties} therefore tells us $\ell_{\infty}(r) = w_{n-1} - 2^{v_2(r)}$.
Furthermore,
\begin{align*}
m(r) - \ell_{\infty}(b) + \ell_1(b) &<2^{v_2(b) +1} - w_{n-1} + w_{n+1} - 3 \cdot 2^{v_2(b)}\\
\intertext{and thus}
m(r) + m(b) &\leq w_{n-1} - 2^{v_2(b)+1}.\\
\intertext{Together, Lemma \ref{twisted}, Lemma \ref{lemma_remainder_properties}, and Equation \ref{long_equation} tell us that}
\ell_{\infty}(u_b b - u_r r) &\leq \max (\ell_{\infty}(b) - \ell_{\infty}(r), m(b) + m(r)) \leq w_{n-1} - 2^{v_2(b) +1}\\
\intertext{and}
\ell_1(u_b b - u_r r) &= \ell_{\infty}(b) + m(b) - \ell_{\infty}(r) + m(r) \\
&\leq w_{n+1} - 3 \cdot 2^{v_2(b)} - (w_{n-1} - 2^{v_2(r)}) + (w_n - w_{n-1} - 2^{v_2(r)})\\
&= w_n - 3 \cdot 2^{v_2(b)}.
\end{align*}
We conclude that 
$\phi_{\Z[i]} \left ( r - \frac{u_b}{u_r} b \right ) = \phi_{\Z[i]} (u_b b - u_r r) < \phi_{\Z[i]}(b)$.
\end{proof}

We now put all our lemmas together to prove Theorem \ref{main_result}.
\begin{proof}
Proposition \ref{im_align} proves our claim when condition 1 holds, Lemma \ref{m(r)+m(b)_leq} proves it when condition 2 holds, and Lemma \ref{ell+m>} proves it 
when condition 3 holds.

If $\Im(u_b b) \Im(u_r r) <0$ and neither condition 2 nor condition 3 hold, then either $m(b) \geq \ell_{\infty}(r)$ or 
$m(b) < \ell_{\infty}(r) < m(b) +m(r)$ and $\ell_{\infty}(b) - m(r) \leq w_{n-1} -2^{v_2(b) +1}$.
Lemmas \ref{m(b)_leq} proves our theorem in the first situation, and Lemma \ref{ell+m_leq} proves it in the second.
\end{proof}

\section{Acknowledgements}

I thank my family for their patience with me when I had the proof's necessary insight on our vacation, and my colleagues Jon Grantham and Tad White for our fruitful discussions.


\begin{thebibliography}{99}



\bibitem{Graves} H. Graves,
\newblock{``The  Minimal Euclidean Function on the Gaussian Integers,''}
\newblock{\sl Indag. Math. (N.S.)}, 34 (2023), no. 1, 78-88. 
\newblock{ https://arxiv.org/abs/2110.13112}




\bibitem{Lenstra} H.W. Lenstra, Jr.,
\newblock{``Lectures on Euclidean Rings,''}
\newblock{Bielefield}, 1974.\\
\newblock{https://www.math.leidenuniv.nl/$\sim$lenstrahw/PUBLICATIONS/1975b/art.pdf}

\bibitem{Motzkin} T. Motzkin, 
\newblock{``The Euclidean Algorithm,''}
\newblock{ \sl Bull. Am. Math. Soc}, 55 (1949), 1142-1146.


 
\end{thebibliography}
\end{document}